\documentclass[10pt, twoside]{scrartcl}
\usepackage{etex}
\usepackage[a4paper, top=88pt, bottom=88pt, left=72pt, right=72pt, headsep=16pt, footskip=28pt]{geometry}
\usepackage{amsmath, amssymb, amsthm, xcolor, tabularx}
\usepackage[hyphens]{url}
\usepackage[colorlinks, citecolor=blue, linkcolor=blue, urlcolor=blue]{hyperref}
\usepackage[nameinlink]{cleveref}
\usepackage{enumitem}
\usepackage{mleftright} \mleftright
\renewcommand{\qedsymbol}{$\blacksquare$}
\renewenvironment{proof}[1][\proofname]{\noindent{\bfseries #1.} }{\hfill \qedsymbol \medskip}

\renewcommand*{\paragraph}[1]{\smallskip\noindent\textbf{\textsf{#1}}\,\,}
\newcommand*{\arXiv}[1]{\bgroup\color{blue}\href{https://arxiv.org/abs/#1}{arXiv:#1}\egroup}
\newcommand*{\doi}[1]{\bgroup\color{blue}\href{https://doi.org/#1}{doi:#1}\egroup}
\newcommand*{\email}[1]{\bgroup\color{blue}\href{mailto:#1}{#1}\egroup}
\renewcommand*{\url}[1]{\bgroup\color{blue}\href{#1}{#1}\egroup}

\usepackage{scrlayer-scrpage, xhfill}
\automark[section]{section}
\setkomafont{pageheadfoot}{\normalcolor\sffamily}
\setkomafont{pagenumber}{\normalfont\normalsize\sffamily}
\clearpairofpagestyles
\let\oldtitle\title
\renewcommand{\title}[1]{\oldtitle{#1}\newcommand{\theshorttitle}{#1}}
\newcommand{\shorttitle}[1]{\renewcommand{\theshorttitle}{#1}}
\let\oldauthor\author
\renewcommand{\author}[1]{\oldauthor{#1}\newcommand{\theshortauthor}{#1}}
\newcommand{\shortauthor}[1]{\renewcommand{\theshortauthor}{#1}}
\cohead{\xrfill[0.525ex]{0.6pt}~\theshorttitle~\xrfill[0.525ex]{0.6pt}}
\cehead{\xrfill[0.525ex]{0.6pt}~\theshortauthor~\xrfill[0.525ex]{0.6pt}}
\usepackage{natbib}
\setcitestyle{round}
\usepackage{mathtools}
\usepackage{graphicx}
\usepackage[abs]{overpic}
\usepackage{latexsym}
\usepackage{algorithm}
\usepackage{algorithmic}
\newcommand{\<}{\langle}
\renewcommand{\>}{\rangle}
\newcommand{\R}{\mathbb{R}}

\newcommand{\D}{\mathcal{D}}
\newcommand{\G}{\Gamma}
\newcommand*{\defeq}{\coloneqq}
\newcommand{\smk}{\mathcal{S}_m^k}
\newcommand{\simk}{\Sigma_m^k}
\newcommand{\simkm}{\Sigma_{m,m}^k}
\newcommand{\rmk}{\mathbb{R}_m^k}
\newcommand{\ga}{\gamma}
\newcommand{\vphi}{\varphi}
\newcommand*{\SO}{\mathrm{SO}}
\newcommand*{\T}{\mathrm{T}}
\newcommand{\Log}{\mathrm{Log}}
\newcommand{\Exp}{\mathrm{Exp}}
\newcommand{\ver}{\mathrm{ver}}
\newcommand{\hor}{\mathrm{hor}}
\newcommand{\Ver}{\mathrm{Ver}}
\newcommand{\Hor}{\mathrm{Hor}}
\newcommand{\rd}{\mathrm{d}}
\newcommand{\rD}{\mathrm{D}}
\newcommand{\rR}{\mathrm{R}}
\newcommand{\cA}{\mathcal{A}}
\newcommand{\cT}{\mathcal{T}}
\newcommand{\Skew}{\mathrm{Skew}}
\newcommand{\w}{\stackrel{\omega}{\sim}}
\newcommand{\ti}[1]{\tilde{#1}}
\newcommand{\wti}[1]{\widetilde{#1}}

\DeclareMathOperator{\argmin}{arg\,min}

\DeclareMathOperator{\rank}{rank}
\DeclareMathOperator{\tr}{trace}
\DeclareMathOperator{\var}{var}
\newcommand*{\quark}{\setbox0\hbox{$x$}\hbox to\wd0{\hss$\cdot$\hss}}

\usepackage{xcolor}
\newcommand{\todo}[1]{\bgroup\color{black}#1\egroup}
\newcommand{\revised}[1]{{\leavevmode\color{black}#1}}
\newcommand{\revisedx}[1]{{\leavevmode\color{black}#1}}

\DeclareRobustCommand{\ShowColormap}{\raisebox{-0.14em}{\includegraphics[height=.8em,trim={119px 0 0 0},clip]{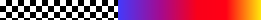}}}

\newtheorem{theorem}{\sffamily Theorem}[section]

\newtheorem{lemma}[theorem]{\sffamily Lemma}
\newtheorem{proposition}[theorem]{\sffamily Proposition}

\theoremstyle{definition}

\newtheorem{remark}[theorem]{\sffamily Remark}

\newcommand*{\affilref}[1]{\ref{affiliation#1}}
\newcommand*{\affiliation}[3]{
	\footnotetext[#1]{\label{affiliation#2} #3}
}

\begin{document}

\title{Geodesic analysis in Kendall's shape space\\with epidemiological applications\footnote{This is a preprint version of an article to appear in the \emph{Journal of Mathematical Imaging and Vision} {\copyright} Springer with the DOI \href{https://doi.org/10.1007/s10851-020-00945-w}{10.1007/s10851-020-00945-w} and differs from the final published version in layout and  typographical detail.}}
\shorttitle{Geodesic analysis in Kendall's shape space with epidemiological applications}

\author{%
	Esfandiar Nava-Yazdani\textsuperscript{\affilref{ZIB}}
	\and
	Hans-Christian Hege\textsuperscript{\affilref{ZIB}}
	\and
	T.~J.~Sullivan\textsuperscript{\affilref{ZIB},\affilref{FUB}}
	\and
	Christoph von Tycowicz\textsuperscript{\affilref{ZIB}}
}
\shortauthor{E.\ Nava-Yazdani, H.-C.\ Hege, T.\ J.\ Sullivan, and C.\ von Tycowicz}

\date{\today}

\maketitle

\affiliation{1}{ZIB}{Zuse Institute Berlin, Takustra{\ss}e 7, 14195 Berlin, Germany. (\email{navayazdani@zib.de}, \email{hege@zib.de}, \email{sullivan@sib.de}, \newline \email{vontycowicz@zib.de})}
\affiliation{2}{FUB}{Freie Universit{\"a}t Berlin, Arnimallee 6, 14195 Berlin, Germany. (\email{t.j.sullivan@fu-berlin.de})}

\begin{abstract}
\paragraph{Abstract:}
We analytically determine Jacobi fields and parallel transports and compute geodesic regression in Kendall's shape space.
Using the derived expressions, we can fully leverage the geometry via Riemannian optimization and thereby reduce the computational expense by several orders of magnitude \revisedx{over common, nonlinear constrained approaches}.
The methodology is demonstrated by performing a longitudinal statistical analysis of epidemiological shape data.
As an example application we have chosen 3D shapes of knee bones, reconstructed from image data of the Osteoarthritis Initiative (OAI).
Comparing subject groups with incident and developing osteoarthritis versus normal controls, we find clear differences in the temporal development of femur shapes.
This paves the way for early prediction of incident knee osteoarthritis, using geometry data alone.

\end{abstract}

\section{Introduction}
\label{sec:intro}

In recent years, there has been an increased interest in statistical analysis of geometric shapes.
Such analyses are especially often performed in the field of morphometry, but mostly for static forms.
A frequently-encountered situation, however, is that instead of a set of discrete shapes, series of shapes are given, often together with co-varying parameters.
For example, longitudinal imaging studies track biological shape changes over time within and across individuals to gain insight into dynamical processes such as ageing or disease progression.
Statistical modeling and analysis of shapes is of critical importance for \revised{a} better understanding of such temporal shape data.


The main challenge is that shape variability is inherently nonlinear and high-dimensional, so that classical statistical approaches are not always appropriate.
One way to address this is linearization.
The quality of the resulting statistical model, however, then depends strongly on the validity of the linearity assumption, i.e.\ that the observed data points lie to a good approximation in a flat Euclidean subspace.
Since the natural variability in populations often leads to a large spread in shape space and the observed data may lie in highly-curved regions (see \citet{HH2014}), linearity often cannot be assumed in practical applications.

In the context of longitudinal studies, an important task is to estimate continuous trajectories from sparse and potentially noisy samples.
For smooth individual biological changes, subject-specific spatiotemporal regression models are adequate.
They also provide a way to describe the data at unobserved times (i.e.\ shape changes between observation times and --- within certain limits --- also at future times) and to compare trends across subjects in the presence of unbalanced data (e.g.\ due to drop-outs).
One approach in use is to approximate the observed temporal shape data by geodesics in shape space and, based on these, to estimate overall trends within groups.
Geodesic models are attractive as they feature a compact representation (similar to the slope and intercept term in linear regression) and therefore allow for computationally efficient inference.

\revisedx{The intrinsic theory of least squares and geodesic regression in shape spaces has been introduced in \citet{F2013}.
For the derivation of the corresponding Euler--Lagrange equations for some important manifolds, we refer to \citet{SM2006}.
An extension to intrinsic Riemannian polynomials has been considered in \citet{Hinkle2014}.
Earlier related results in the framework of large deformation diffeomorphic metric mapping (LDDMM) can be found in \citet{qiu2008parallel} and \citet{qiu2009time}. 
In \citet{banerjee2016nonlinear}, the authors present a kernel-based generalization of geodesic regression to manifold-valued longitudinal parameters.
For an overview of statistical analysis on Riemannian manifolds see \citet{HH2014} and \citet{P2006}.}

An additional challenge in the analysis of shape trajectories is to distinguish between morphological differences due to (i) temporal shape evolutions of a single individual and (ii) the geometric variability in a population of an object class under study.
To obtain a statistically significant localization of structural changes at the population level (group-wise statistics), the subject-specific trajectories need to be transferred in a standard reference frame.
Among the different techniques proposed for normalizing longitudinal deformations~\citep{rao2004spatial,bossa2010changing}, constructions based on \emph{parallel transport} provide the most natural approach 
and have shown superior sensitivity and stability in the context of diffeomorphic registration~\citep{Lorenzi2011}.
Note also that, for general trajectories, the simple transport of each shape is not suitable because the distances between the shapes are not preserved.
However, if the shapes belong to the same geodesic, this problem does not arise, which is another advantage of geodesic regression.

As parallel transport in curved shape spaces is rarely given in closed form, in general it has to be approximated numerically, e.g.\ employing Schild's ladder~\citep{Lorenzi2011} \revised{for fanning~\citep{louis2018fanning}}.
For shapes in 2D, Kendall's shape space is isomorphic to the projective space, which is a symmetric space, so that the essential geometric quantities are well known (cf.\ \citet{HHM10} and \citet{F2013}).
However, for three and more dimensions, because of less restrictive structure, many questions remain open.
\revised{Utilizing closed form expressions of the pre-shape sphere, we reduce parallel transport to the solution of a homogeneous first-order differential equation that allows for highly efficient computations.
Moreover, we reduce the important case of parallel transport along a geodesic path to the solution of a low-dimensional equation that only depends on the dimension of the ambient space and not on the spatial resolution of the discrete representation.}



\revised{
The paper is organized as follows.
In \Cref{sec:geometry}, after a short overview of Kendall's shape space, we provide a computationally efficient approach (via the so-called Sylvester equation) for the canonical decomposition of tangent vectors into horizontal and vertical components, which is essential for the geometry and analysis of shapes and trajectories.
Moreover, we determine parallel transport and Jacobi fields, which will be employed for geodesic regression.
Parallel transport is essential for statistical normalization, alignment of trajectories and also computation of Jacobi fields.
The latter describes the variability of trajectories that will be modeled as best-fitting geodesics in \Cref{sec:regression}, where we also present our algorithm for the computation of geodesic regression.
In \Cref{sec:epidemiology} we apply this algorithm to yield longitudinal statistical analysis of femur data from an epidemiological study dealing with osteoarthritis and discuss the numerical results.
}

\section{Geodesic Analysis in Shape Space}
\label{sec:geometry}

A \emph{pre-shape} is a $k$-ad of landmarks (i.e.\ particular points) in $\R^m$ after removing translations and similarity transformations.
A \emph{shape} is a pre-shape with rotations removed.
For a comprehensive introduction to Kendall's shape space and details on the subjects of this section, we refer to \citet{KBC1999}.
For the relevant tools from Riemannian geometry, we refer to \citet{GHL2004}.

\subsection{Shape Space}

In the following we present a brief overview of Kendall's shape space, provide a computationally efficient method to determine horizontal and vertical 
components of tangent vectors of the pre-shape space, and also prove the corresponding equivariance \revised{under rotations}.

Let $x\in M(m,k)$, where $M(m,k)$ denotes the space of real $m\times k$ matrices. \revised{Denoting the columns of $x$ by $x_i$ and their Euclidean mean by $\bar{x}$, in order to remove translations, 
we replace $x_i$ by $x_i-\bar{x}$.} The result $\rmk \defeq \{x\in M(m,k):\: \sum_{i=1}^kx_i=0\}$, identified with $M(m,k-1)$, 
will be endowed with its canonical scalar product given by $\<x,y\>=\tr(xy^t)$.
Denoting the Frobenius norm by $\| \quark \|$, we call the sphere $\smk \defeq \{x\in \rmk:\: \|x\|=1\}$ \emph{pre-shape space} and endow it with the spherical Procrustes metric $d(x,y) \defeq \arccos (\< x,y\>)$.
Now, the left action of $\SO_m$ on $\smk$ given by $(R,x)\mapsto Rx$ defines an equivalence relation given by $x\sim y$ if and only if $y=Rx$ for some $R\in \SO_m$.
\emph{Kendall's shape space} is defined as $\simk=\smk/\mathord{\sim}$.
Provided that $k\geq m+1$, the dimension of $\simk$ is $m(k-1)-\frac{1}{2}m(m-1)-1$.
Now, denoting the canonical projection of $\sim$ by $\pi$, the induced distance between any two shapes $\pi(x)$ and $\pi (y)$ is given by
\[
	d_{\Sigma}(x,y) \defeq \min_{R\in \SO_m} d(x,Ry)=\arccos \sum_{i=1}^m\lambda_i
\]
where $\lambda_1\geq \cdots\geq \revised{\lambda_{m-1}\geq}  |\lambda_m|$ denote the pseudo-singular values of $yx^t$.
Denoting $\D_j \defeq \{x\in \smk:\:\rank(x)\leq j\}$, it turns out that $\simkm \defeq \simk\setminus \pi(\D_{m-2})$ inherits a differential structure that is compatible with its quotient topology.
Following \citet{KBC1999}, we refer to $\pi(\D_{m-2})$ as the singular part of $\simk$.
In particular, $\simk$ is a strata of manifolds with varying dimensions and $\simkm$ is open and dense in $\simk$.
Away from the singular part, the quotient map $\pi$ is a Riemannian submersion \revised{with respect to the metric induced by the ambient Euclidean space}.
Moreover, for $k\geq 3$, the shape space $\Sigma_1^k$ (resp.\ $\Sigma_2^k$) is isometric to the sphere (resp.\ projective space).
We call $x,y\in \smk$ \emph{well positioned}, and write $x\w y$, if and only if $yx^t$ is symmetric and $d(x,y)=d_{\Sigma}(x,y)$.
For each $x,y\in \smk$, there exists an optimal rotation $R\in \SO_m$ such that $x\w Ry$.
Note that $R$ does not need to be unique.
Let \revised{$\mathbb{U}$} denote a neighborhood in $\smk$ with radius smaller then $\pi/4$ (the diameter of $\simk$ is $\pi/2$) such that
\begin{align*}
	\lambda_{m-1}+\lambda_m> 0\:\text{ for all }x,y\in \mathbb{U}.
\end{align*}
For $x,y\in \mathbb{U}$ the optimal rotation $R$ is unique and the function
\[
\smk\ni y\mapsto\omega(x,y) \defeq Ry
\]
is well-defined.

\revised{We recall that, for a Riemannian submersion $f \colon M\to N$ and $y\in N$, $f^{-1}(y)$ is a submanifold of $M$. 
For any $x\in M$, denoting the kernel of $\rd_x f$ by $\Ver_x$, the tangent space $\T_x M$ to $M$ at $x$ admits an orthogonal decomposition 
$\T_x M=\Hor_x\oplus \Ver_x$ where $\Hor_x$ and the $\Ver_x$ are the so-called \emph{horizontal} and \emph{vertical} subspaces.} Due to \citet{KBC1999} 
the vertical space at $x\in \smk$ is given by
\[
	\Ver_x=\{Ax:\:A+A^t=0\} ,
\]
and the horizontal space is given by
\[
	\Hor_x=\{u\in M(m,k-1):\:ux^t=xu^t\text{ and } \<x,u\>=0\}.
\]
We denote the vector space of $m\times m$ skew-symmetric real matrices by $\Skew_m$.
Thus $\Ver_x=\Skew_m\cdot x$.
\revised{Furthermore, a smooth curve is called \emph{horizontal} if and only if its tangent field is horizontal.
Geodesics in the shape space are equivalence classes of horizontal geodesics.}
Now, let $\exp$ and $\log$ denote the exponential and logarithm map of the pre-shape space.
For $x\w y$ the geodesic from $x$ to $y$ given by
\begin{equation}
	\label{eqphi}
	\Phi(t,x,y) \defeq \exp_x(t\log_x y)=\frac{\sin((1-t)\vphi)}{\sin\vphi}x+\frac{\sin(t\vphi)}{\sin\vphi}y
\end{equation}
with $\vphi=\arccos(\<x,y\>),\:0\leq t\leq 1$, is horizontal.
Hence $\Phi$ realizes the minimizing geodesic from $\pi(x)$ to $\pi(y)$.
The following result concerns determination and $\SO_m$-equivariance for horizontal and vertical projection.
\begin{lemma}
	\label{lemprj}
	Fix $x\in \smk$ and $w\in \T_x\smk$.
	Let $\ver_x$ resp. $\hor_x$ denote the restriction of vertical resp. horizontal projection to $\T_x\smk$.
	\begin{enumerate}
		\renewcommand{\labelenumi}{(\alph{enumi})}
		\item $\ver_x(w)=Ax$ if and only if $A$ solves the Sylvester equation
		\begin{equation}
			\label{eqsylvester}
			Axx^t+xx^tA=wx^t-xw^t.
		\end{equation}
		\revised{Moreover, the above equation has a unique skew-symmetric solution if $\rank(x)\geq m-1$}.
		\item Fix $R\in \SO_m$.
		Then $\ver_{Rx}(Rw)=R\ver_x(w)$ \revised{and $\hor_{Rx}(Rw)=R\hor_x(w)$}.
	\end{enumerate}
\end{lemma}

\begin{proof}
	For (a), let $\ver_x(w)=Ax$, i.e., $w=u+Ax$ with $ux^t$ symmetric and $A\in \Skew_m$.
	A straightforward computation eliminating $ux^t$ implies that \eqref{eqsylvester} holds.
	To prove the converse, let $j \defeq \rank(x)$.
	Suppose without loss of generality that $j>1$ and write $x=\binom{x_1}{0}$ with
	\[
		\rank(x_1)=j,\:w=\binom{w_1}{w_0} ,
	\]
	where $w_1$ is $j\times k$.
	\revised{We observe that both equations $A_1x_1x_1^t+x_1x_1^tA_1=w_1x_1^t-x_1w_1^t$ and $a^tx_1x_1^t=-w_0x_1^t$ are uniquely solvable, since $x_1x_1^t$ is invertible.
	Furthermore, the solution of the first equation is skew-symmetric, since its right-hand side is skew-symmetric. It follows that
	\[
	A=\begin{pmatrix}A_1 & a\\ -a^t & A_0\end{pmatrix}
	\]
	with $A_0\in \Skew_{m-j}$ arbitrary, is skew-symmetric and solves the Sylvester equation \eqref{eqsylvester} which also implies that
	$(w-Ax)x^t$ is symmetric. Hence $Ax$ is the vertical component of $w$.
	If $\rank(x)=m-1$, then $A_0=0$. If $x$ has full rank, then $A=A_1$}.
	
	For (b) note that $\<Rw,Rx\>=\<w,x\>=0$, i.e., $w\in \T_xS$ implies $Rw\in \T_{Rx}S$.
	Now, $\ver_{Rx} (Rw)=BRx$ where $B$ is the solution of \revised{$BRxx^tR^t+Rxx^tR^tB=R(wx^t-xw^t)R^t$}.
	Hence $B=RAR^t$, which implies that $\ver_{Rx}(Rw)=R.\ver_x (w)$ \revised{and $\hor_{Rx}(Rw)=R\hor_x(w)$}.
\end{proof}

Henceforth the superscript $v$ (resp.\ $h$) denotes the vertical (resp.\ horizontal) component, i.e., for any $w\in\rmk$ we have the orthogonal decomposition $w=\<w,x\>x+w^h+w^v$.
Due to the explicit computation above, $(R.w)^v=R.w^v$ and $(R.w)^h=R.w^h$, i.e., horizontal and vertical projections are $\SO_m$-equivariant.
Note that this property holds even if $\pi (x)$ belongs to the singular part of the shape space.
As appropriate for our applications and for brevity, unless otherwise specified, we restrict our data to the
open and dense set $S \defeq \{x\in \smk:\:\rank(x)\geq m-1\}$ on which $\pi$ is a Riemannian submersion, thus the geometry of the shape space is mainly described by its horizontal lift in the pre-shape space.
In particular, for $x\in S$ the Sylvester equation \eqref{eqsylvester} has a unique solution determining horizontal and vertical projections and the restriction of $\rd_x\pi$ to $\Hor_x$ 
is an isometry of Euclidean vector spaces $\Hor_x$ and $\T_{\pi(x)}\simkm$.
Denoting the covariant derivatives in the pre-shape and shape space by $\nabla$ resp.\ $\ti{\nabla}$, for horizontal vector fields $X$ and $Y$ we have
\[
	(\ti{\nabla}_{\rd \pi X} \rd \pi Y) \circ \pi = \rd \pi(\nabla_X Y).
\]
In the following $[ \quark , \quark ]$ denotes the Lie bracket in $\rmk$, i.e., $[U,V]= \rD V(U)- \rD U(V)$ ($\rD$ Euclidean).
For the Euclidean derivative of a vector field $W$ along a curve $\ga$ in $\rmk$ we
use $\frac{\rD}{\rd t}$ and also for simplicity of notation a dot, i.e., $\nabla_{\dot{\ga}}W=\dot{W}-\<\dot{W},\ga\>\ga$ if $\|\ga\|=1$, and
$\frac{\rD^2W}{\rd t^2}=\ddot{W}$, etc.
We set\footnote{Note that the Riemannian exponential map of the shape space denoted by $\wti{\exp}$ satisfies
$\pi(\exp_x u)=\wti{\exp}_{\pi(x)}(\rd_x\pi(u))=\wti{\exp}_{\pi(x)}(\rd_x\pi(u^h))$.}
\[
	\Log_xy \defeq \log_x\omega(x,y),\:\Exp_xu \defeq \exp_x u^h,\:u\in \T_x\smk.
\]
For the computation of the Fr{\'e}chet mean (cf.~\citet{HHM10} and \citet{P2006}) $\pi(\bar{q})$ of the
shapes $\pi(q_1),\cdots,\pi(q_N)$ with $q_i\in \revised{\mathbb{U}}$, i.e.,
\begin{align}
	\label{eq:mean_objective}
	\bar{q} & \defeq \argmin_{x} G(x),\: G(x)  \defeq \sum_{i=1}^N d_{\Sigma}^2(x,q_i) ,
\end{align}
we apply Newton's method to Karcher's equation $\sum_{i=1}^N\Log_xq_i=0$ as follows.
We search for the unique zero $\bar{q}$ of the function $f$ defined by
\[
	f(x)=\sum_{i=1}^N\Log_xq_i,\:x\in \mathbb{U} ,
\]
and set
\[
	x_{k+1}=\Exp_{x_k}(-(\rd_{x_k}f)^{-1}f(x_k)).
\]
A suitable initial value is the normalized Euclidean mean
\[
	x_0=\frac{1}{\|\sum_{i=1}^Nq_i\|}\sum_{i=1}^Nq_i.
\]
The total variance of $q=(q_1,\cdots,q_N)$ reads
\[
	\var(q)=\frac{1}{N}G(\bar{q})=\frac{1}{N}\sum_{i=1}^N\|\Log_{\bar{q}}q_i\|^2.
\]

\subsection{Parallel Transport}
Next, we \revised{derive formulas for} parallel transport in the shape space and its relation to parallel transport in the pre-shape space.\footnote{Essentially, part (a) of \Cref{proppt} was recently also obtained by \citet{KDL2018}.}

We call a vector field $W$ along a horizontal curve $\ga$ \emph{horizontally parallel} (for brevity \emph{h-parallel}) if and only if $W$ is horizontal
and $\rd \pi W$ is parallel along $\pi\circ \ga$. In the following, we derive the differential equation for the h-parallelism of $W$ and a corresponding
constructive approach using a Sylvester equation in certain cases.
\begin{proposition}
	\label{proppt}
	Let $\ga :[0,\tau]\to S$ be a smooth horizontal curve with initial velocity $v$, $u$ a horizontal vector at $x \defeq \ga(0)$ and $W$ a vector field along $\ga$ with $W(0)=u$.
	\begin{enumerate}
		\renewcommand{\labelenumi}{(\alph{enumi})}
		\item The vector field $W$ is h-parallel transport of $u$ \revisedx{along $\ga$} if and only if $\dot{W}=A\ga-\<W,\dot{\ga}\>\ga$ where $A$ is the unique solution of
		\begin{equation}
			\label{eqpt1}
			A\ga\ga^t+\ga\ga^tA=\dot{\ga}W^t-W\dot{\ga}^t.
		\end{equation}
		\item Suppose that $\ga$ is a unit-speed geodesic.
		Then equation \eqref{eqpt1} reduces to
		\begin{equation}
			\label{eqpt2}
			\dot{A}\ga\ga^t+\ga\ga^t\dot{A}+3(A\dot{\ga}\ga^t+\ga\dot{\ga}^tA)=0.
		\end{equation}
		\item Let $Cv$ denote the orthogonal projection of $u$ on $\Skew_m\cdot v$, i.e.\ $Cvv^t+vv^tC=uv^t-vu^t$.
		Suppose that $C\dot{\ga}$ is horizontal. If $\ga$ is a unit-speed geodesic, then the h-parallel transport of $u$ is given by
		\begin{equation}
			\label{eqptmain}
			W=U+(\<u,v\>+C)(\dot{\ga}-v)
		\end{equation}
		where $U$ denotes the Euclidean parallel extension of $u$ along $\ga$, i.e., $U(t)=u$ f.a.\ $t$. If $y = \ga (\varphi)$ with $\varphi=d (x,y)$, then
		the h-parallel transport $W_y$ of $u$ along $\ga$ to $y$ reads
		\begin{equation}
			\label{eqptmain1}
			W_y = U- 2\frac{\<u,y\>+C\sin(\varphi)}{\|x+y\|^2}(x+y)
		\end{equation}
	\end{enumerate}
\end{proposition}
\begin{proof}
	(a) \revisedx{$\rd \pi W$ is parallel along $\pi\circ\ga$} if and only if $\rd \pi(\nabla_{\dot{\ga}}W)=0$, i.e., infinitesimal variation of $W$ must be vertical.
	Hence $\nabla_{\dot{\ga}}W=(\nabla_{\dot{\ga}}W)^v$, which due to \Cref{lemprj} equals
	$A\ga$ with $A\ga\ga^t+\ga\ga^t A=(\nabla_{\dot{\ga}}W)\ga^t-\ga(\nabla_{\dot{\ga}}W)^t=\dot{W}\ga^t-\ga\dot{W}^t$.
	Moreover, $\SO_m$-equivariance of vertical projection implies the well-definedness, i.e., if $\rd \pi W$ is parallel, then $\rd \pi(Rw)$
	is parallel for all $R\in \SO_m$.
	Note that existence and uniqueness of the solution for \revised{\eqref{eqpt1}} with $W(0)=u$ is immediate from the existence and uniqueness of parallel transport and vertical projection.
	Now, $W$ is horizontal if and only if $\dot{f}=0$ where $f \defeq \|W\ga^t-\ga W^t\|^2+\<W,\ga\>^2$, since $f(0)=0$.
	If \revisedx{equation} \eqref{eqpt1} \revisedx{holds},
	then
	\[
		\dot{W}\ga^t-\ga\dot{W}^\revised{t} = (\nabla_{\dot{\ga}}W)\ga^t-\ga(\nabla_{\dot{\ga}}W)^t = A\ga\ga^t+\ga\ga^tA=\dot{\ga}W^t-W\dot{\ga}^t
	\]
	and
	\[
		\<\dot{W},\ga\>+\<W,\dot{\ga}\> = \<A\ga-\<W,\dot{\ga}\>\ga,\ga\>+\<W,\dot{\ga}\> = \<A\ga,\ga\>=0.
	\]
	\revised{The last equation follows from the fact that $A$ is skew-symmetric and $\ga\ga^t$ is symmetric, hence their product is trace-free.} Now, we arrive at
	$f=0$, i.e., $W$ remains horizontal. To prove the converse, note that if $W$ is horizontal, then $f$ and therefore $\dot{f}$ vanishes. Hence
	$\dot{W}\ga^t-\ga\dot{W}^t=\dot{\ga}W^t-W\dot{\ga}^t$ and $\<W,\dot{\ga}\>+\<\dot{W},\ga\>=0$ and the Sylvester equation for the
	vertical component of $\dot{W}$ reads $A\ga\ga^t+\ga\ga^tA=\dot{\ga}W^t-W\dot{\ga}^t$.
	Thus \eqref{eqpt1} follows.
	
	(b) Note that $W\ga^t$ and $\dot{\ga}\ga^t$ are symmetric and $\ddot{\ga}+\ga=0$. \revised{Thus $W\ddot{\ga}^t=-W\ga^t$ is also symmetric.}
	Now, \eqref{eqpt1} implies
	\begin{align*}
		\dot{A}\ga\ga^t+\ga\ga^t\dot{A}+2(A\dot{\ga}\ga^t+\ga\dot{\ga}^tA)
		&=\ddot{\ga}W^t-W\ddot{\ga}^t+\dot{\ga}\dot{W}^t-\dot{W}\dot{\ga}^t\\
		&=\dot{\ga}(A\ga-\<W,\dot{\ga}\>\ga)^t -(A\ga-\<W,\dot{\ga}\>\ga)\dot{\ga}^t\\
		&=-(A\ga\dot{\ga}^t+\dot{\ga}\ga^tA).
	\end{align*}
	
	(c) Obviously $W$ given by \eqref{eqptmain} satisfies the initial condition $W(0)=u$. Moreover, it satisfies $\dot{W}=-C\ga-\<W,\dot{\ga}\>\ga$, i.e.,
	\eqref{eqpt2} holds with \revised{$A(t)=-C$}. To prove \eqref{eqptmain1}, insert
	$v=\frac{1}{\varphi}\log_xy=\frac{y-x\cos (\varphi )}{\sin(\varphi)}$ and $\dot{\ga}=\frac{-1}{\varphi}\log_yx$ into \eqref{eqptmain}.
\end{proof}

Note that, \revised{due to skew-symmetry of $\ga\ga^t(\nabla_{\dot{\ga}}W)\ga^t$}, the differential equation for the h-parallel transport can also be written as
\begin{equation}
	(\nabla_{\dot{\ga}}W)\ga^t\ga\ga^t+\ga\ga^t(\nabla_{\dot{\ga}}W)\ga^t=(\dot{\gamma}W^t-W\dot{\gamma}^t)\ga\ga^t.
\end{equation}
Hence, a vector field along a curve in $\pi(S)$ is parallel if and only if it has a horizontal lift satisfying the above equation.\\

\begin{remark}
	We mention two cases such that \eqref{eqptmain} and \eqref{eqptmain1} apply. First, $W$ coincides with the spherical parallel transport of $u$ if and only if $uv^t=vu^t$ or, equivalently, $C=0$. Secondly, for planar shapes. To see this, let $\chi_i$ and $\eta_i$ denote the rows of a shape $\chi$ and $\eta$ a horizontal vector
	at $\chi$. Fix $\mu\in\R$ and let $C \defeq \mu \begin{pmatrix}0& 1\\-1& 0\end{pmatrix}$.
	Then $C\eta \chi^t=\mu\begin{pmatrix}\eta_2\chi_1^t& \eta_2\chi_2^t\\ -\eta_1\chi_1^t& -\eta_1\chi_2^t\end{pmatrix}$ is symmetric since $\<\eta,\chi\>=0$.
	Hence, $C\dot{\ga}$ is horizontal for arbitrary $C\in \Skew_2$.
\end{remark}

\subsection{Jacobi Fields}
\revised{Next, we derive the differential equation for Jacobi fields and
provide a constructive approach to its solution utilizing parallel
transport.}

\revised{We recall that a smooth horizontal curve $\ga$ in $S$ is a
geodesic if and only if $\pi\circ \ga$ is a geodesic in $\pi(S)$.}
Hence, for a horizontal geodesic $\ga$, any geodesic variation of
$\pi\circ \ga$ in the latter space reads $\pi\circ \G$ with $\G$ a
variation of $\ga$ through horizontal geodesic.
Thus the variation field $\frac{\rd}{\rd s}(\pi\circ \G(s, \quark
))|_{s=0}=\rd \pi ( \frac{\rd}{\rd s}\G(s, \quark )|_{s=0})$ is a Jacobi
field of the shape space.
Recall that a vector field $J$ along $\ga$ is called \emph{normal} if and
only if $\<J,\dot{\ga}\>=0$ and the tangential component of any Jacobi
field is
just given by $(a+bt)\dot{\ga}(t)$ with $a,b\in\R$, which is obviously
horizontal.
Thus the challenge is to find those normal vector fields that project to a
Jacobi field in the shape space.
\revised{
\begin{theorem}
	Let $J$ be a normal vector field along $\ga$ and denote
	\[
        K = \left(\frac{\rD J^v}{\rd t}\right)^v+2\left( \frac{\rD
J^h}{\rd t} \right)^v.
    \]
	\begin{enumerate}
		\renewcommand{\labelenumi}{(\alph{enumi})}
		\item $\rd \pi (J)$ is a Jacobi field if and only if
		\begin{align}
			\label{eqjf1}\left( \frac{\rD^2J}{\rd t^2} + J\right)^h &=
2\left(\frac{\rD K}{\rd t} \right)^h\\
			\label{eqjf12}
			\left( \frac{\rD^2J}{\rd t^2} + J\right)^v &= \left(\frac{\rD K}{\rd t}
\right)^v
		\end{align}
		\item A normal Jacobi field $J^S$ of the pre-shape sphere projects to a
Jacobi field if and only if $\left(\frac{\rD K}{\rd t} \right)^h=0$.
	\end{enumerate}
\end{theorem}

\begin{proof}
	(a) Obviously, solutions of the equations are, due to
$\SO_m$-equivariance of horizontal and vertical projection, invariant
under $\SO_m$ action.
    Let $Y$ and $Z$ be vector fields on $\smk$. Following \citet{Oneill66},
the $\cA$ and $\mathcal{T}$-tensor fields are defined as
	\begin{align*}
	 \cT_YZ &=(\nabla_{Y^v}Z^v)^h+(\nabla_{Y^v}Z^h)^v,\\
	 \cA_YZ &=(\nabla_{Y^h}Y^v)^h+(\nabla_{Y^h}Z^h)^v.
	\end{align*}
	Due to \cite[Theorem 2]{Oneill67}, $\rd \pi(J)$ is a Jacobi field if and
only if
	\begin{align*}
	 \left(\frac{\rD^2J}{\rd t^2}  -\rR (J,X,X) \right)^h &=2\cA_X K,\\
	 \left(\frac{\rD^2J}{\rd t^2}-\rR (J,X,X) \right)^v &=\left(\frac{\rD
K}{\rd t}\right)^v+\cT_K X,
	\end{align*}
	where $X=\dot{\ga}$, $\frac{\rD}{dt}$ stands for $\nabla_X$ and the
vector field $K$ is given by
	\[
        K(J)=\left(\frac{\rD J^v}{\rd t}\right)^v-\cT_{J^v}X+2\cA_X J^h.
	\]
	Note that as $J$ is normal, its covariant and Euclidean derivative
coincide. In our setting, the fibers are totally geodesic, hence
$\cT\equiv 0$. Therefore
    $K = \left(\frac{\rD J^v}{\rd t}\right)^v+2\left( \frac{\rD J^h}{\rd
t} \right)^v $. Moreover, we may suppose $\|X\|=1$. Thus
$\rR(J,X,X)=-J$ and we arrive at \eqref{eqjf1} and \eqref{eqjf12}.

	(b) It follows immediately from $\frac{\rD^2J^S}{\rd t^2}+J^S=0$.
\end{proof}

In the following, we give a geometric construction for normal Jacobi
fields which will be employed for geodesic regression.  
\revisedx{For $\xi\in T_xS$, we set $ver_{x,v}(\xi):=Ax$, where $A$ denotes the solution (cf.
\Cref{lemprj}) of the Sylvester equation \[Axx^t+xx^tA=v\xi^t-\xi v^t.\]}
In the sequel $x=\ga (0)$, $v=\dot{\ga}(0)$, $\xi_1,\xi_2\in Hor_x$, $\langle \xi_1,v\rangle =\langle \xi_2,v\rangle =0$ and $i=1,2$.
\begin{theorem}
	\label{thjf2}
	Suppose that $\ga$ has unit speed $v$. Let $J$ be the solution of the
	differential equations \eqref{eqjf1}, \eqref{eqjf12} with $J^v(0)=Cx$, 
	$J^h(0)=\xi_1$ and $\frac{\revisedx{\rD} J^h }{d t}(0)=\xi_2$. Then $J^v=C\ga$ and the
	following hold.
	\begin{enumerate}
		\renewcommand{\labelenumi}{(\alph{enumi})}
		\item Let $Z_i$ and $Y_1$ denote the parallel extensions of $\xi_i$ resp. $Ax=ver_{x,v}(\xi_1)$ along $\ga$. Then \[J^h(t)=(\cos (t)Z_1(t)-\sin (t)(Y_1(t)+Z_2(t)))^h.\]
		\item Suppose that $m=2$. Let $w_i$ denote the orthogonal projection of $\xi_i$ tangent to $\Skew_2\cdot v$ and $u_i$ its orthogonal complement. Then   
		\begin{align*}
			J^h=\cos (t) U_1(t)+\cos (2t) W_1(t)+\sin(t)U_2(t)+\frac{1}{2}\sin(2t)W_2(t),
		\end{align*}
		where $W_i$ and $U_i$ are parallel extensions of $w_i$ and $u_i$.
	\end{enumerate}
\end{theorem}

\begin{proof}
We may write $\ga (t)=\cos(t)x+\sin(t)v$ with $\|v\|=1$, $\<x,v\>=0$ and
$vx^t=xv^t$. The variation
\[
 \G(s,t)=\exp(sC)\ga (t)
\]
is a variation of $\ga$ through horizontal geodesics since $\dot{\G}\G^t$ is symmetric, and defines a vertical Jacobi field $J^v=\frac{d\G}{ds}(0,.)=C\ga$ with $J(0)=Cx$.

(a) As any Jacobi field is a linear combination of parallel vector fields, and those vector fields conserve \revisedx{orthogo}nality
and length, we may assume $\revisedx{\|}\xi_i\|=1$. Furthermore, due to $\langle \xi_i,v\rangle=0$ and $\langle Ax,v\rangle =0$, $Z_i(t)=\xi$ and $Y(t)=Ax$. Now, let $V$ denote the h-parallel transport of $v$ along the geodesic $\alpha$ given by $\alpha(s)=\cos (s) x+\sin (s) \xi_1$. Due to \eqref{eqpt2}, we have $V^\prime (0)=Ax$. \revisedx{Now, consider the variations of $\ga$ given by}
\[
 \G_1(s,t)=\cos (t)\alpha (s) +\sin (t) V(s)
\]
and
\[
 \G_2(s,t)=\cos (t)x +\sin (t)(\cos (s)v+\sin (s)\xi_2).
\]
A straightforward computation shows that $\dot{\G_i}\G_i$ is symmetric, i.e., $\G_i$ is a variation of $\ga$ through horizontal
geodesics. Hence $d\pi J_i$ is a Jacobi field, where $J_i=\frac{d \G_i }{d s}(0,.)$. Therefore, $d\pi (J_1+J_2)$ is a Jacobi field. Moreover, $\frac{d \G_2 }{d s}(0,0)=0$ and $\frac{\rD  }{d t}\frac{d \G_2 }{d s}(0,0)=\frac{\rD }{d s}\frac{d \G_2 }{d t}(0,0)=\xi_2$. 
Hence the solution with $J^h(0)=0$ and $\frac{\revisedx{\rD} J^h }{d t}(0)=\xi_2$ is given by $J_2^h$, where $J_2(t)=\sin (t)Z_2(t)=\frac{d \G_2 }{d s}(0,t)$. It follows that the solution with $J^h(0)=0$ and $\frac{\rD J^h }{d t}(0)=Ax$ is given by $(t\mapsto \sin (t)Y_1(t))^h$. 
Furthermore, $\frac{d \G_1 }{d s}(0,0)=\xi_1$ and $\frac{\rD  }{d t}\frac{d \G_2 }{d s}(0,0)=\frac{\rD }{d s}\frac{d \G_2 }{d t}(0,0)=Ax$. The fact that the space of horizontal vector fields along $\ga$ is linear, completes the proof.

(b) Let $Q=\begin{pmatrix}0& 1\\ -1 & 0\end{pmatrix}$. Then $\xi_i$ enjoys the orthogonal decomposition $\xi_i=u_i+w_i$, where $w_i=B_iv$, $u_iv^t=vu_i^t$ and $B_ivv^t+vv^tB_i=\xi_i v^t-v\xi_i^t$. Moreover, $A=\mu Q$ and $B_i=\lambda_i Q$ for some $\mu, \lambda_i\in \R$.
A straightforward computation shows $Qxx^t+xx^tQ=Q=Qvv^t+vv^tQ$ (note that $\|x\|=\|v\|=1$). Hence $A=-B_1$. Now, the vector $w_i$ is horizontal (cf. Remark 1) and normal. 
Hence the vector $u_i=\xi_i-w_i$ is also horizontal and normal. Therefore its parallel extension is given by $U_i(t)=u_i$. Using $u_iv^t=vu_i^t$, we arrive at $U_i \ga^t=\ga U_i^t$, i.e. $U_i^h=U_i$. Moreover, utilizing the fact that $Q^2$ is minus indentity, we have $W_i^h=W_i-W_i^v=B_iv-\langle B_iv,Q\ga\rangle Q\ga =\lambda_i(Qv-\langle Qv,Q\ga\rangle Q\ga)=\cos (t) B_i\dot{\ga}=\cos (t)W_i$. Similarily, for the constant vector field $B_1x$, we have $(B_1x)^h=-\sin (t)W_1$. Implying in the expression of $J^h$ from part a), we arrive at the desired formula.
\end{proof}

}

We recall that for $m=2$, the shape space is isometric to the complex projective space endowed with 
its standard (Fubini--Study) metric. The given formula for $J^h$ in this case is well-known (cf.\ \citet{F2013} and \citet{J2017}).
\section{Geodesic Regression}
\label{sec:regression}

In the following, we employ the results of the previous section to derive an efficient and robust approach for finding the relation between an independent scalar variable, i.e.\ time, and a dependent shape-valued random variable.

Regression analysis is a fundamental tool for the spatiotemporal modeling of longitudinal observations.
Given scalars $t_1 < t_2 < \cdots < t_N$ and distinct pre-shapes $q_1,\cdots,q_N$, the goal of geodesic regression is to find a geodesic curve in shape space that best fits the data in a least-squares sense.
In particular for a horizontal geodesic $\ga$ from $x$ to $y$ with $v=\dot{\ga}(0)$, we define the misfit between the data and the geodesic as a sum of squared distances with respect to $d_\Sigma$, i.e.
\begin{equation}
	\label{eq:regression_objective}
	F(\ga) \defeq \sum_{i=1}^Nd_\Sigma^2(q_i,\ga(t_i)).
\end{equation}
We can assume that $t_1=0$ and $t_N=1$.
While the authors of \cite{F2013} and \cite{SM2006} identify geodesics by their initial point and velocity --- and hence they consider $F(x,v)$ --- we use for the identification their endpoints, i.e., we consider
\[
	F(x,y) = \sum_{i=1}^Nd_\Sigma^2(q_i,\ga(t_i)) = \sum_{i=1}^Nd_\Sigma^2(q_i,\Phi(t_i,x,y)).
\]
The reason is that geodesic computations in terms of the function $\Phi$ defined in equation \eqref{eqphi}\revised{, the so-called slerp (spherical linear interpolation),} are more efficient. 
Model estimation is then formulated as the least-squares problem
\[
	(x^\ast,y^\ast) = \argmin_{(x,y)} F(x,y),\: x\w y.
\]

In the absence of an analytic solution, the regression problem has to be solved numerically.
To this end, we employ a Riemannian trust-regions solver~\citep{manopt} with a Hessian approximation based on finite differences and use
$(q_1,\omega(q_1,q_N))$ as initial guess.
\revised{Having in mind that (cf.\ \citet{P2006} and \citet{J2017})
\begin{equation}
\label{eq:gradient}
   \nabla \rho_y(x) = -2\Log_x y=-2(\log_xy)^h
\end{equation}
where $\rho_y(x) \defeq d_\Sigma^2(x,y)$, the gradient of the cost function $F$ can be computed using Jacobi fields,
since they express the derivatives of the exponential map and therefore those of $\Phi$. 
Now, for fixed $q$ and $t$, let $\nabla_xf$ denote the gradient of $f$ with respect to $x$ where $f(x,y):=\rho_q\circ\Phi(t,x,y)$. 
Then for any $u\in\Hor_x$, $d_x\Exp_xtv\cdot u=J(t)$ where $d\pi J$ is the horizontal Jacobi field along $\ga$ with $J(0)=u$ and $J(1)=0$. 

\revisedx{In the following, $\top$ and $\perp$ denote tangent resp. orthogonal components of vectors. Now, let $\alpha$ denote the unit speed horizontal geodesic from $y$ to $x$, i.e., $\alpha (s)=\cos (s) y+\sin (s)v$ with $v=\log_yx$, $s\in [0,\varphi]$ and $\varphi=\|v\|$. Denoting the horizontal component of the \revisedx{parallel} extension of $u$ along $\ga$ by $U$, and $\tilde{J}(s)=\frac{\sin s}{\sin \varphi} (U^\perp)^h$, due to \revisedx{T}heorem \ref{thjf2}, $d\pi \tilde{J}$ is a Jacobi field with $\tilde{J}(0)=0$ and $\frac{1}{\sin \varphi}\dot{\tilde{J}}(0)=(U^\perp)^h$. Reparametrization only changes the tangent component of the Jacobi field. Moreover, horizontal projection does not depend on the parametrization.} Due to the fact that $t\mapsto \Phi(t,x,y)$ parametrizes the reverse geodesic by arc length ($\|\dot{\Phi}(0,x,y)\|=\varphi$), we arrive at $J(t)=\tilde{J}((1-t)\varphi)$. Hence
\[
 J(t)=\frac{\sin((1-t)\varphi)}{\sin\varphi}(U^\perp)^h+(1-t)U^\top.
\]
Let $P_x$ denote the h-parallel transport to $x$ along $\ga$. 
In view of \eqref{eq:gradient}, $-\frac{1}{2}P_{\ga(t)}\nabla_xf(x,y)$ is the adjoint of the mapping $u\mapsto J(t)$. 
As the latter is self-adjoint it follows that
\[
 \nabla_xf(x,y)=-2P_x \left( \frac{\sin((1-t)\varphi)}{\sin\varphi}(W^\perp)^h+(1-t)W^\top \right) ,
\]
where $W=\Log_{\ga (t)}q$.
To get the gradient of $f$ with respect to $y$, we simply replace $1-t$ by $t$ (another advantage of employing the parametrization \eqref{eqphi}) and arrive at
\begin{align*}
 \nabla_xF(x,y)&=-2P_x\sum_{i=1}^N \left( \frac{\sin((1-t_i)\varphi)}{\sin\varphi}(W_i^\perp)^h+(1-t_i)W_i^\top \right)\\
 \nabla_yF(x,y)&=-2P_y\sum_{i=1}^N \left( \frac{\sin(t_i\varphi)}{\sin\varphi}(W_i^\perp)^h+t_iW_i^\top \right) ,
\end{align*}
where $W_i=\Log_{\ga (t)}q_i$. 

Now, our procedure for geodesic regression can be summarized as presented in Algorithm \ref{alg:reg}.

\begin{algorithm}
\caption{Geodesic regression in shape space}
\label{alg:reg}
\begin{algorithmic}
\REQUIRE Pre-shapes $q_1,\cdots,q_N$ and time instances $t_1\cdots,t_N$
\ENSURE Minimizer $(x^\ast,y^\ast)$
\STATE Initialize: $(x_0,y_0) \leftarrow (q_1,\omega(q_1,q_N))$
\STATE Define cost function $F$ and its gradient $gradF$
\STATE Create the problem structure $P$:
\STATE $P.manifold \leftarrow Sphere(m,k)$
\STATE $P.cost \leftarrow @(x,y) F(x,\omega (x,y))$
\STATE $P.grad \leftarrow @(x,y) grad F(x,\omega (x,y))$
\STATE Minimize: $(x^\ast,y^\ast,cost) \leftarrow Solver(P,x_0,y_0)$
\end{algorithmic}
\end{algorithm}
}
\section{Application to Epidemiological Data}
\label{sec:epidemiology}

In this section, we analyze the morphological variability in longitudinal data of human distal femora in order to quantify shape changes that are associated with femoral osteoarthritis.

\subsection{Data Description}

We apply the derived scheme to the analysis of group differences in longitudinal femur shapes of subjects with incident and developing osteoarthritis (OA) versus normal controls.
\revised{An overview of OA-related dysmorphisms is shown in \Cref{fig:healthyVsOA}.}
The dataset is derived from the Osteoarthritis Initiative (OAI), which is a longitudinal study of knee osteoarthritis maintaining (among others) clinical evaluation data and radiological images from 4,796 men and women of age 45--79.
The data are available for public access at \url{http://www.oai.ucsf.edu/}.
\begin{figure}[ht]
	\centering
	\includegraphics[width=0.8\textwidth]{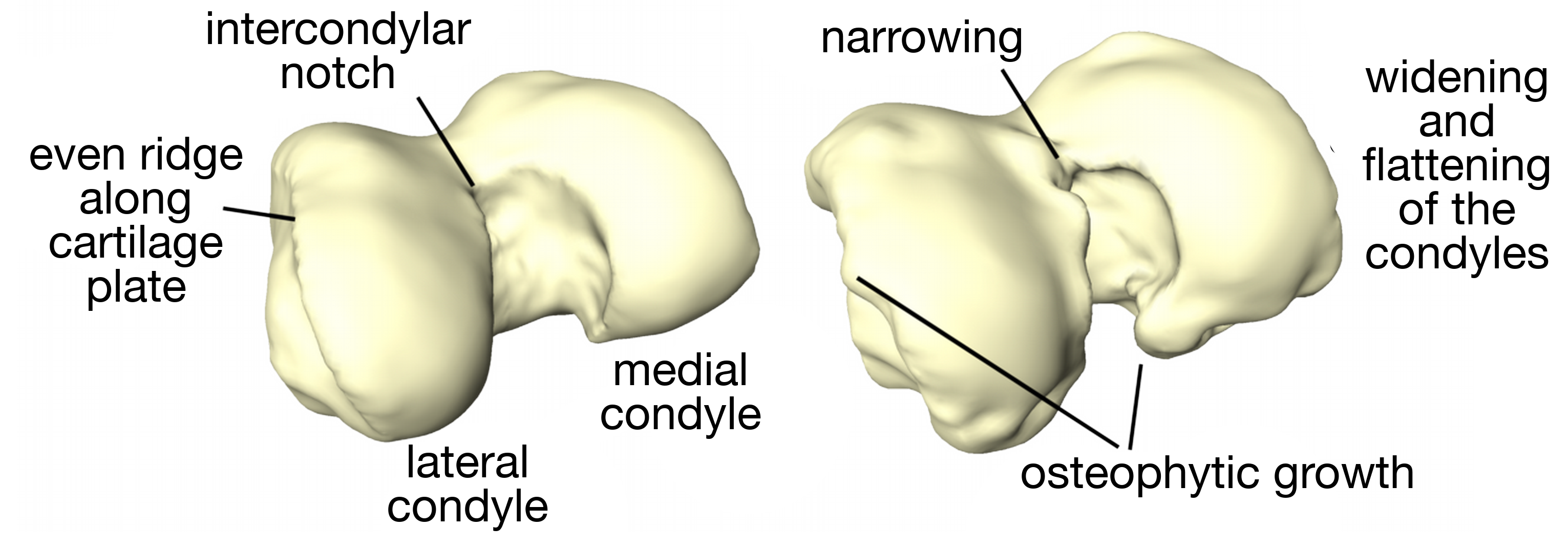}
	\caption{Healthy (left) and osteoarthritic (right) distal femur with delineated pathological changes in shape.}
	\label{fig:healthyVsOA}
\end{figure}
From the OAI database, we determined three groups of shapes trajectories: HH (healthy, i.e.\ no OA), HD (healthy to diseased, i.e.\ onset and progression to severe OA), and DD (diseased, i.e.\ OA at baseline) according to the Kellgren--Lawrence 
score~\citep{kellgren1957KLscore} of grade 0 for all visits, an increase of at least 3 grades over the course of the study, and grade 3 or 4 for all visits, respectively.
We extracted surfaces of the distal femora from the respective 3D weDESS MR images (0.37$\times$0.37 mm matrix, 0.7 mm slice thickness) using a state-of-the-art automatic segmentation approach \citep{Ambellan2018segmentation}.
For each group, we collected 22 trajectories (all available data for group DD minus a record that exhibited inconsistencies, and the same number for groups HD and HH, randomly selected), each of which comprises shapes of all acquired MR images, 
i.e.\ at baseline, the 12-, 24-, 36-, 48- and 72-month visits.
In a supervised post-process, the quality of segmentations as well as the correspondence of the resulting meshes (8,988 vertices) were ensured.

\subsection{Geodesic Modeling of Femoral Trajectories}

We apply the geodesic regression approach detailed in \Cref{sec:regression} to the femoral shape trajectories described above and represented in Kendall's shape space.
Due to the expressions derived for the parallel transport and Jacobi fields, we can fully leverage the geometry using 
Riemannian optimization procedures \revised{(cf. \citet{AMS2007}).
In particular, for the intrinsic treatment of the optimization problem underlying the geodesic regression we use the open-source Matlab toolbox manopt \citep{manopt}.
In our experiments}, we observed a superlinear convergence of the intrinsic trust-region solver for most of the shape trajectories.
Solving the high-dimensional (54k degrees of freedom) regression problem on a laptop computer with Intel Core i7-7500U ($2\times{}2.70$GHz) CPU took about 0.3s on average.
In contrast, the generic Matlab \revised{routine for nonlinear regression (viz. \emph{fitnlm})} required about 25s to determine a solution, thus being two orders of magnitude slower.

The resulting estimated geodesics along with the original trajectories are visualized in \Cref{fig:regtrjs}.
\revised{
The geodesic representation provides a less cluttered visualization of the trajectory population
making it easier to identify trends within as well as across groups.
For 2-dimensional visualization we perform dimension reduction for the trajectories $X^1,\cdots,X^k$ with $X^j=(x_1^j,\cdots,x_n^j)$, i.e. we apply tangent PCA to $(x_i^j)_{i=1,\cdots,n}^{j=1,\cdots,k}$ at the mean of all baseline shapes in HH. In the remainder, the latter is referred to as reference shape Ref.
}

\begin{figure*}[h]
 \begin{center}
    \includegraphics[width=.9\textwidth]{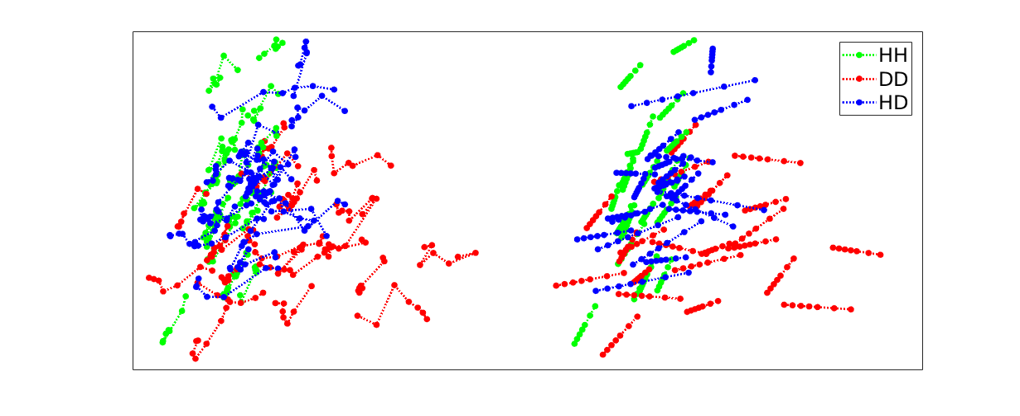}
	\caption{\revised{Principal components} for femoral shape trajectories of subjects with
	no (HH), progressing (HD), and severe (DD) osteoarthritis (left) and their qualitatively estimated shape trajectories via geodesic
	regression (right).
	Note that points on the left show the observed shapes, while those on the right show the corresponding points on the fitted geodesic.
	}
\label{fig:regtrjs}
 \end{center}
\end{figure*}

Next we would like to answer the question of how well the observed data is replicated by the estimated geodesic trends.
A common approach to test this is to compute the \emph{coefficient of determination}, denoted as $R^2$, that is the proportion of the total variance in the data explained by the model.
Following \citet{F2013}, a generalization to manifolds is defined as
\[
	R^2 = 1 - \frac{\text{unexplained variance}}{\text{total variance}} = 1 - \frac{\min_{\ga} F(\ga)}{\min_x G(x)},
\]
with $F(\ga)$ and $G(x)$ as defined in equations \eqref{eq:regression_objective} and \eqref{eq:mean_objective}, respectively.
As the unexplained variance cannot exceed the total variance (since the Fr{\'e}chet mean lies in the search space of the regression problem) and both variances are nonnegative, $R^2$ must 
lie in the interval $\left[0,1\right]$ (with larger values indicating a higher proportion of the variance being explained by the model).

The coefficients of determination were computed for all estimated trends amounting to
\revised{group-wise medians (95\% confidence intervals) of 0.40 (0.33--0.46), 0.55 (0.48--0.63), and 0.51 (0.40--0.72)} for group HH, DD, and HD, respectively.
While for all groups the geodesic model is able to describe a relatively large portion of the shape variability, there is a \emph{clear difference} between the control group HH and the groups DD and HD associated to osteoarthritis.
In particular, pairwise Mann--Whitney U tests confirm that the differences are highly unlikely due to random chance (with $p$-values of \revisedx{$<10^{-3}$}, and $0.005$ for HH vs.\ DD, and HH vs.\ HD, respectively).
These findings indicate that the OA-related shape \revised{variability is better caputured by a single variable (time) than the physiological trends in HH}.
Based on the coefficient of determination we also test for the significance of the estimated trends employing permutation tests as suggested in \citet{F2013}.
For each of the trajectories we performed 1,000 permutations and considered the results as statistically significant for $p$-values less than $0.01$.
In almost all cases (63 out of 66) the trends were significant, such that we can expect them to be highly unlikely due to random chance.

\subsection{Group-wise Analysis of Longitudinal Trends}

\revised{
\begin{algorithm}
\caption{Transport of shape trajectory}
\label{alg:transport}
\begin{algorithmic}
\REQUIRE Pre-shapes $x_1,\cdots,x_n,\:Ref$
\ENSURE Transported pre-shapes $y_1,\cdots,y_n$
\STATE $y_1 \leftarrow Ref$
\FOR{$k = 1,\cdots,n-1$}
\STATE $v_k \leftarrow \Log_{x_k}x_{k+1}$
\STATE $y_{k+1} \leftarrow \Exp(y_k,ParTrans(y_k,v_k))$
\ENDFOR
\end{algorithmic}
\end{algorithm}
}

\revised{
In order to perform group-wise analysis of longitudinal shape changes we compare the estimated geodesic trends of the femoral trajectories.
This requires the consistent integration of intra- and inter-subject variability in order to obtain statistically significant localization of changes at the population level.
In fact, the comparison of longitudinal shape changes is usually performed after normalizing (i.e.\ transporting) them into a common system of coordinates (see~\citet{lorenzi2014poleladder} and the references therein).
Such a normalization can be realized by adapting parallel transport as presented in Algorithm~\ref{alg:transport}.
In particular, for geodesic trends this scheme reduces to parallel transport of the subject-specific velocity along the baseline-to-reference shape geodesic.
The group-wise longitudinal progression was modeled as the mean of the \revisedx{transported velocities}.
The areas of significant differences between longitudinal changes were investigated by two-sample Hotelling's $T^2$ tests on the vertex-wise displacements corresponding to the transported velocity-fields.
While the displacements differ significantly ($p<0.01$, \revisedx{ after Benjamini-Hochberg false discovery correction}) between normal controls and the OA groups (for \revisedx{55\% and 19\%} of the vertices for HH vs.\ HD and HH vs.\ DD, respectively), there are no differences in the longitudinal changes in-between both OA groups.
\Cref{fig:hotelling} shows a qualitative visualization of the group tests in terms of the magnitude of the difference between the group-wise means.
Visible are changes along the ridge of the cartilage plate (characteristic regions for osteophytic growth, cf.\ \Cref{fig:healthyVsOA}) in comparison of both HH vs.\ HD as well as HH vs.\ DD, albeit the latter are less pronounced suggesting a saturation of morphological developments.
Additionally, the changes are more developed on the medial compartment, which is in line with previous findings \citep{vincent2012pathophysiology}.
}
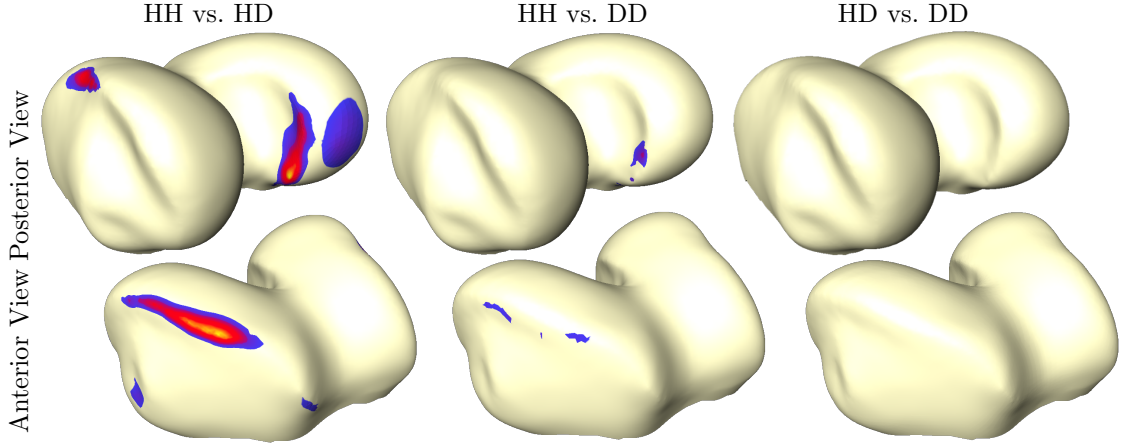
\begin{figure}[t]
 \begin{center}
 	\begin{overpic}[width=.9\textwidth]{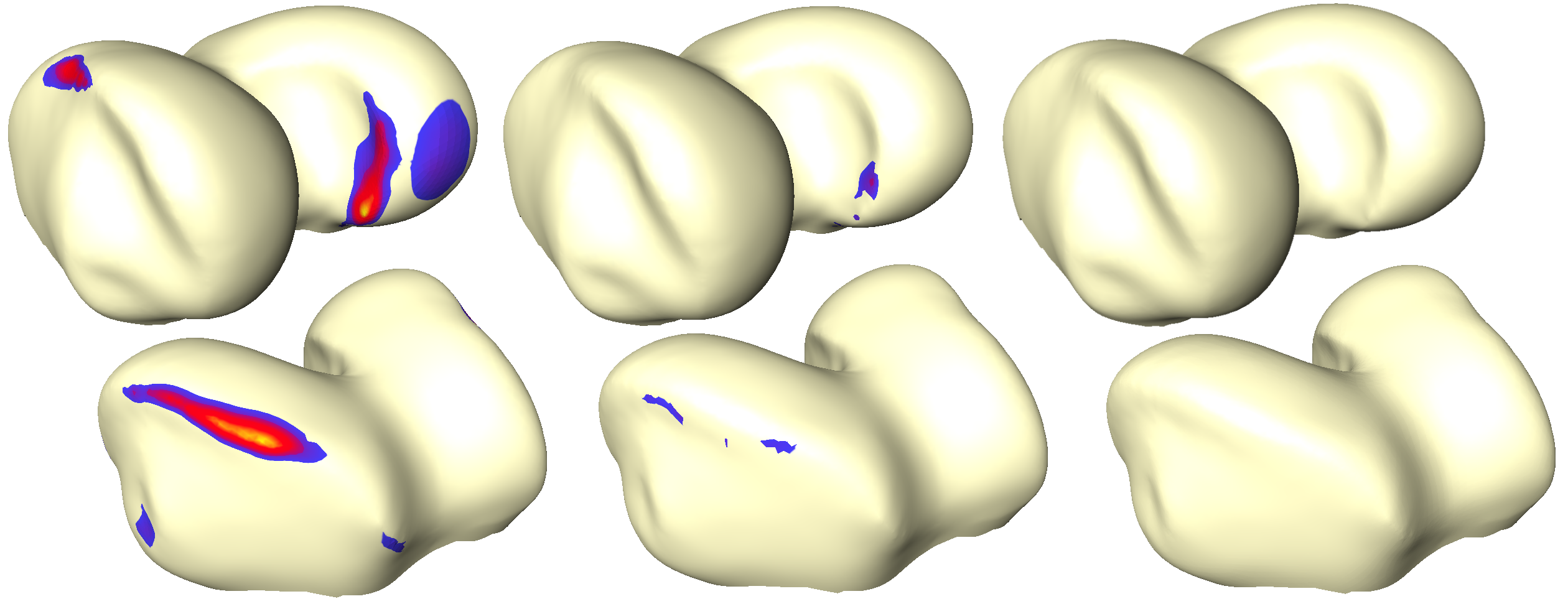}
		\put(-10,5){\rotatebox{90}{Anterior View}}
		\put(-10,70){\rotatebox{90}{Posterior View}}
		\put(40,160){HH vs.\ HD}
		\put(180,160){HH vs.\ DD}
		\put(300,160){HD vs.\ DD}
	\end{overpic}
	\caption{\revised{
	Group-wise analysis of femoral geodesic trends: Magnitude of differences between the group-average trends for HH vs.\ HD (left column), HH vs.\ DD (middle column), and HD vs.\ DD (right column) after transport to common reference shape. Only significantly different displacements ($p<0.01$) are shown (2.0e-4~\ShowColormap~4.2e-4).
	}}
 \label{fig:hotelling}
 \end{center}
\end{figure}
\revised{While velocities are constant for subject-specific geodesics, their parallel transport depends on the path (an effect called holonomy).
To investigate this path dependency, we repeated the above experiment using different paths for the HD group.
In particular, we chose the shape at the onset time (transition time to severe OA, viz.\ Kellgren--Lawrence score $\geq$ 3) as the initial point for the transport path.
In line with previous work \citep{lorenzi2013geodesics}, we found that the results are not sensitive to the path.
More precisely, the results of the group tests agreed for \revisedx{99.70\% and 100.00\%} of the vertices for HH vs.\ HD and HD vs.\ DD, respectively.
}

\section{Concluding Remarks}

This work presented characterizations of and computationally efficient methods for the determination of parallel transport, Jacobi fields and geodesic regression of data represented as shapes in Kendall's space.
Furthermore, an application to longitudinal statistical analysis of epidemiological data (femur data for analysis of knee osteoarthritis) has been shown.
An advantage of modeling trajectories by geodesics is the following: A main task
in longitudinal analysis is to translate trajectories to start at a reference shape.
The intermediate distances between the shapes of a geodesic are preserved by parallel transport, which is not the case for general transports.
Moreover, data inconsistencies are minimized by considering the best-fitting geodesics, and Jacobi fields can be employed to analyze the variability of the geodesics, hence providing a canonical descriptor of trends and differences for the trajectories.

There are many potential avenues for future work.
\revised{First, we would like to use the presented methodology within the mixed-effect framework (see e.g.~\citet{bone2018learning}), which provides a joined analysis of longitudinal and cross-sectional variability.}
In particular, group-wise means of the geodesics can be computed with respect to a natural metric in the tangent bundle (e.g.\ the Sasaki metric) to determine the group parameters as described in~\citet{MF2012}.
Second, an extension of the method to higher-dimensional longitudinal parameters instead of just time can be examined, to achieve even more differentiated results.
Third, spline regression poses a natural generalization providing more degrees of freedom.

\revised{On the application side, based on the results found, it can be said in summary that the shape trajectories
of the healthy subjects expose significantly different temporal changes than those found in groups with incident and developing OA.
Our analysis delivered detailed insights into the complex morphological changes that fit medical knowledge.}
It seems possible to make a correct assignment to one of the three groups based on just two measurements.
The aim of further investigations must be to substantiate this statement, by determining with what reliability a prediction can be made about the onset of knee osteoarthritis depending on
the baseline shape and trend as well as the sensitivity of the latter with respect to the number of observations made and the time intervals between them.

\section*{Acknowledgements}
\addcontentsline{toc}{section}{Acknowledgements}
This research was carried out in the framework of \textsc{Matheon} supported by the Einstein Foundation Berlin (E.~Nava-Yazdani),
DFG project \emph{In-vivo soft tissue kinematics from dynamic MRI} (C.~von Tycowicz), and the Freie Universit{\"a}t Berlin through the Excellence Initiative of the Deutsche Forschungsgemeinschaft (T.~J.~Sullivan).
Furthermore we are grateful for the open-access dataset provided by the OAI\footnote{The Osteoarthritis Initiative is a public-private partnership comprised of five contracts
(N01-AR-2-2258; N01-AR-2-2259; N01-AR-2-2260; N01-AR-2-2261; N01-AR-2-2262)
funded by the National Institutes of Health, a branch of the Department of Health and
Human Services, and conducted by the OAI Study Investigators.
Private funding partners
include Merck Research Laboratories; Novartis Pharmaceuticals Corporation,
GlaxoSmithKline; and Pfizer, Inc.
Private sector funding for the OAI is managed by the
Foundation for the National Institutes of Health.
This manuscript was prepared using an OAI
public use data set and does not necessarily reflect the opinions or views of the OAI investigators, the NIH, or the private funding partners.}
as well as for the open-source software Manopt~\citep{manopt}.

\bibliographystyle{abbrvnat}
\bibliography{shape_trj.bib}   

\begin{thebibliography}{28}
\providecommand{\natexlab}[1]{#1}
\providecommand{\url}[1]{\texttt{#1}}
\expandafter\ifx\csname urlstyle\endcsname\relax
  \providecommand{\doi}[1]{doi: #1}\else
  \providecommand{\doi}{doi: \begingroup \urlstyle{rm}\Url}\fi

\bibitem[Absil et~al.(2007)Absil, Mahony, and Sepulchre]{AMS2007}
P.-A. Absil, R.~Mahony, and R.~Sepulchre.
\newblock \emph{{Optimization Algorithms on Matrix Manifolds}}.
\newblock Princeton University Press, Princeton, NJ, USA, 2007.
\newblock URL \url{http://www.manopt.org}.

\bibitem[Ambellan et~al.(2018)Ambellan, Tack, Ehlke, and
  Zachow]{Ambellan2018segmentation}
F.~Ambellan, A.~Tack, M.~Ehlke, and S.~Zachow.
\newblock Automated segmentation of knee bone and cartilage combining
  statistical shape knowledge and convolutional neural networks: Data from the
  {Osteoarthritis} {Initiative}.
\newblock In \emph{Medical Imaging with Deep Learning}, 2018.
\newblock URL \url{https://doi.org/10.1016/j.media.2018.11.009}.

\bibitem[Banerjee et~al.(2016)Banerjee, Chakraborty, Ofori, Okun, Viallancourt,
  and Vemuri]{banerjee2016nonlinear}
M.~Banerjee, R.~Chakraborty, E.~Ofori, M.~S. Okun, D.~E. Viallancourt, and
  B.~C. Vemuri.
\newblock A nonlinear regression technique for manifold valued data with
  applications to medical image analysis.
\newblock In \emph{Proceedings of the IEEE Conference on Computer Vision and
  Pattern Recognition}, pages 4424--4432, 2016.
\newblock URL \url{https://doi.org/10.1109/CVPR.2016.479}.

\bibitem[B{\^o}ne et~al.(2018)B{\^o}ne, Colliot, and
  Durrleman]{bone2018learning}
A.~B{\^o}ne, O.~Colliot, and S.~Durrleman.
\newblock Learning distributions of shape trajectories from longitudinal
  datasets: a hierarchical model on a manifold of diffeomorphisms.
\newblock In \emph{Proceedings of the IEEE Conference on Computer Vision and
  Pattern Recognition}, pages 9271--9280, 2018.

\bibitem[Bossa et~al.(2010)Bossa, Zacur, and Olmos]{bossa2010changing}
M.~N. Bossa, E.~Zacur, and S.~Olmos.
\newblock On changing coordinate systems for longitudinal tensor-based
  morphometry.
\newblock In \emph{Spatiotemporal Image Analysis for Longitudinal and
  Time-Series Image Data}, 2010.
\newblock CD publication.

\bibitem[Boumal et~al.(2014)Boumal, Mishra, Absil, and Sepulchre]{manopt}
N.~Boumal, B.~Mishra, P.-A. Absil, and R.~Sepulchre.
\newblock {Manopt}, a {Matlab} toolbox for optimization on manifolds.
\newblock \emph{J. Machine. Learn. Res.}, 15:\penalty0 1455--1459, 2014.
\newblock URL \url{http://www.manopt.org}.

\bibitem[Fletcher(2013)]{F2013}
P.~T. Fletcher.
\newblock Geodesic regression and the theory of least squares on {Riemannian}
  manifolds.
\newblock \emph{Int. J. Comp. Vis.}, 105\penalty0 (2):\penalty0 171--185, 2013.
\newblock URL \url{https://doi.org/10.1007/s11263-012-0591-y}.

\bibitem[Gallot et~al.(2005)Gallot, Hulin, and Lafontaine]{GHL2004}
S.~Gallot, D.~Hulin, and J.~Lafontaine.
\newblock \emph{{Riemannian Geometry}}.
\newblock Springer, third edition, 2005.
\newblock URL \url{https://doi.org/10.1007/978-3-642-18855-8}.

\bibitem[Hinkle et~al.(2014)Hinkle, Fletcher, and Joshi]{Hinkle2014}
J.~Hinkle, P.~T. Fletcher, and S.~Joshi.
\newblock Intrinsic polynomials for regression on riemannian manifolds.
\newblock \emph{Journal of Mathematical Imaging and Vision}, 50\penalty0
  (1):\penalty0 32--52, Sep 2014.
\newblock ISSN 1573-7683.
\newblock URL \url{https://doi.org/10.1007/s10851-013-0489-5}.

\bibitem[Huckemann and Hotz(2014)]{HH2014}
S.~Huckemann and T.~Hotz.
\newblock On means and their asymptotics: circles and shape spaces.
\newblock \emph{J. Math. Imaging. Vis.}, 50\penalty0 (1-2):\penalty0 98--106,
  2014.
\newblock URL \url{https://doi.org/10.1007/s10851-013-0462-3}.

\bibitem[Huckemann et~al.(2010)Huckemann, Hotz, and Munk]{HHM10}
S.~Huckemann, T.~Hotz, and A.~Munk.
\newblock Intrinsic shape analysis: {Geodesic} {PCA} for {Riemannian} manifolds
  modulo isometric {L}ie group actions.
\newblock \emph{Stat. Sinica}, 20\penalty0 (1):\penalty0 1--58, 2010.

\bibitem[Jost(2017)]{J2017}
J.~Jost.
\newblock \emph{{Riemannian Geometry and Geometric Analysis}}.
\newblock Springer, seventh edition, 2017.
\newblock URL \url{https://doi.org/10.1007/978-3-319-61860-9}.
\newblock (pp.~251 ff.).

\bibitem[Kellgren and Lawrence(1957)]{kellgren1957KLscore}
J.~H. Kellgren and J.~S. Lawrence.
\newblock Radiological assessment of osteo-arthrosis.
\newblock \emph{Ann. Rheum. Diseases}, 16\penalty0 (4):\penalty0 494, 1957.

\bibitem[Kendall et~al.(1999)Kendall, Barden, Carne, and Le]{KBC1999}
D.~G. Kendall, D.~Barden, T.~K. Carne, and H.~Le.
\newblock \emph{{Shape and Shape Theory}}.
\newblock John Wiley \& Sons, 1999.
\newblock URL \url{https://doi.org/10.1002/9780470317006}.

\bibitem[Kim et~al.(2018)Kim, Dryden, and Le]{KDL2018}
K.-R. Kim, I.~L. Dryden, and H.~Le.
\newblock Smoothing splines on {Riemannian} manifolds, with applications to 3d
  shape space, 2018.
\newblock URL \url{https://arxiv.org/abs/1801.04978}.

\bibitem[Lorenzi and Pennec(2013)]{lorenzi2013geodesics}
M.~Lorenzi and X.~Pennec.
\newblock Geodesics, parallel transport \& one-parameter subgroups for
  diffeomorphic image registration.
\newblock \emph{Int. J. Comp. Vis.}, 105\penalty0 (2):\penalty0 111--127, 2013.
\newblock URL \url{https://doi.org/10.1007/s11263-012-0598-4}.

\bibitem[Lorenzi and Pennec(2014)]{lorenzi2014poleladder}
M.~Lorenzi and X.~Pennec.
\newblock Efficient parallel transport of deformations in time series of
  images: from {Schild}'s to pole ladder.
\newblock \emph{J. Math. Imag. Vis.}, 50\penalty0 (1-2):\penalty0 5--17, 2014.
\newblock URL \url{https://doi.org/10.1007/s10851-013-0470-3}.

\bibitem[Lorenzi et~al.(2011)Lorenzi, Ayache, and Pennec]{Lorenzi2011}
M.~Lorenzi, N.~Ayache, and X.~Pennec.
\newblock Schild's ladder for the parallel transport of deformations in time
  series of images.
\newblock In \emph{Information Processing in Medical Imaging: 22nd
  International Conference, IPMI 2011, Kloster Irsee, Germany, July 3-8, 2011.
  Proceedings}, pages 463--474. Springer Berlin Heidelberg, 2011.
\newblock URL \url{https://doi.org/10.1007/978-3-642-22092-0\_38}.

\bibitem[Louis et~al.(2018)Louis, Charlier, Jusselin, Pal, and
  Durrleman]{louis2018fanning}
M.~Louis, B.~Charlier, P.~Jusselin, S.~Pal, and S.~Durrleman.
\newblock A fanning scheme for the parallel transport along geodesics on
  {Riemannian} manifolds.
\newblock \emph{SIAM J. Num. Anal.}, 56\penalty0 (4):\penalty0 2563--2584,
  2018.
\newblock URL \url{https://doi.org/10.1137/17M1130617}.

\bibitem[Machado and Silva~Leite(2007)]{SM2006}
L.~M. Machado and F.~Silva~Leite.
\newblock Fitting smooth paths on {Riemannian} manifolds.
\newblock \emph{Int. J. Appl. Math. Stat.}, 4\penalty0 (J06):\penalty0 25--53,
  2007.

\bibitem[Muralidharan and Fletcher(2012)]{MF2012}
P.~Muralidharan and P.~T. Fletcher.
\newblock Sasaki metrics for analysis of longitudinal data on manifolds.
\newblock In \emph{2012 {IEEE} Conference on Computer Vision and Pattern
  Recognition, Providence, RI, USA, June 16-21, 2012}, pages 1027--1034, 2012.
\newblock URL \url{https://doi.org/10.1109/CVPR.2012.6247780}.

\bibitem[O'Neill(1966)]{Oneill66}
B.~O'Neill.
\newblock The fundamental equations of a submersion.
\newblock \emph{Mich. Math. J.}, 13\penalty0 (4):\penalty0 459--469, 1966.
\newblock URL \url{https://doi.org/10.1307/mmj/1028999604}.

\bibitem[O'Neill(1967)]{Oneill67}
B.~O'Neill.
\newblock Submersions and geodesics.
\newblock \emph{Duke Math. J.}, 34\penalty0 (2):\penalty0 363--373, 06 1967.
\newblock URL \url{https://doi.org/10.1215/S0012-7094-67-03440-0}.

\bibitem[Pennec(2006)]{P2006}
X.~Pennec.
\newblock Intrinsic statistics on {Riemannian} manifolds: {Basic} tools for
  geometric measurements.
\newblock \emph{J. Math. Imaging. Vis.}, 25\penalty0 (1):\penalty0 127, 2006.
\newblock URL \url{https://doi.org/10.1007/s10851-006-6228-4}.

\bibitem[Qiu et~al.(2008)Qiu, Younes, Miller, and Csernansky]{qiu2008parallel}
A.~Qiu, L.~Younes, M.~I. Miller, and J.~G. Csernansky.
\newblock Parallel transport in diffeomorphisms distinguishes the
  time-dependent pattern of hippocampal surface deformation due to healthy
  aging and the dementia of the alzheimer's type.
\newblock \emph{NeuroImage}, 40\penalty0 (1):\penalty0 68--76, 2008.
\newblock URL \url{https://doi.org/10.1016/j.neuroimage.2007.11.041}.

\bibitem[Qiu et~al.(2009)Qiu, Albert, Younes, and Miller]{qiu2009time}
A.~Qiu, M.~Albert, L.~Younes, and M.~I. Miller.
\newblock Time sequence diffeomorphic metric mapping and parallel transport
  track time-dependent shape changes.
\newblock \emph{NeuroImage}, 45\penalty0 (1):\penalty0 S51--S60, 2009.
\newblock URL \url{https://doi.org/10.1016/j.neuroimage.2008.10.039}.

\bibitem[Rao et~al.(2004)Rao, Chandrashekara, Sanchez-Ortiz, Mohiaddin,
  Aljabar, Hajnal, Puri, and Rueckert]{rao2004spatial}
A.~Rao, R.~Chandrashekara, G.~I. Sanchez-Ortiz, R.~Mohiaddin, P.~Aljabar, J.~V.
  Hajnal, B.~K. Puri, and D.~Rueckert.
\newblock Spatial transformation of motion and deformation fields using
  nonrigid registration.
\newblock \emph{IEEE Trans. Med. Imag.}, 23\penalty0 (9):\penalty0 1065--1076,
  2004.
\newblock URL \url{https://doi.org/10.1109/TMI.2004.828681}.

\bibitem[Vincent et~al.(2012)Vincent, Conrad, Fregly, and
  Vincent]{vincent2012pathophysiology}
K.~R. Vincent, B.~P. Conrad, B.~J. Fregly, and H.~K. Vincent.
\newblock The pathophysiology of osteoarthritis: a mechanical perspective on
  the knee joint.
\newblock \emph{PM\&R}, 4\penalty0 (5):\penalty0 S3--S9, 2012.
\newblock URL \url{https://doi.org/10.1016/j.pmrj.2012.01.020}.

\end{thebibliography}
\addcontentsline{toc}{section}{References}

\end{document}